\documentclass[a4paper]{jpconf}
\usepackage{graphicx}
\usepackage{dcolumn}
\usepackage{bm}

\usepackage{xcolor}
\usepackage{amsthm}
\usepackage{amsmath}
\usepackage{amssymb}
\usepackage{mathrsfs}

\usepackage{empheq}

\usepackage{dcolumn} 
\newcolumntype{d}[1]{D{.}{.}{#1}}

\usepackage{booktabs} 

\usepackage[labelformat=simple]{subcaption}
\usepackage{txfonts}
\usepackage{hyperref}

\theoremstyle{definition}
\newtheorem{definition}{Definition}
\newtheorem{theorem}{Theorem}
\newtheorem{proposition}{Proposition}

\numberwithin{equation}{section}
\numberwithin{definition}{section}
\numberwithin{theorem}{section}
\numberwithin{remark}{section}
\numberwithin{proposition}{section}
\numberwithin{corollary}{section}

\newcommand{\dif}{\text{d}}

\newcommand{\ri}{\text{i}}
\newcommand{\rd}{\text{d}}

\allowdisplaybreaks

\begin{document}
\title{Orthonormal eigenfunction expansions for sixth-order boundary value problems}

\author{N C Papanicolaou$^1$ and I C Christov$^{1,2}$}

\address{$^1$ Department of Computer Science, University of Nicosia, 46 Makedonitissas Avenue, CY-2417 Nicosia, Cyprus}
\address{$^2$ School of Mechanical Engineering, Purdue University, West Lafayette, Indiana 47907, USA}

\ead{papanicolaou.n@unic.ac.cy,christov@purdue.edu}

\begin{abstract}
	Sixth-order boundary value problems (BVPs) arise in thin-film flows with a surface that has elastic bending resistance. To solve such problems, we first derive a complete set of odd and even orthonormal eigenfunctions --- resembling trigonometric sines and cosines, as well as the so-called ``beam'' functions. These functions intrinsically satisfy boundary conditions (BCs) of relevance to thin-film flows, since they are the solutions of a self-adjoint sixth-order Sturm--Liouville BVP with the same BCs. Next, we propose a Galerkin spectral approach for sixth-order problems; namely the sought function as well as all its derivatives and terms appearing in the differential equation are expanded into an infinite series with respect to the derived complete orthonormal (CON) set of eigenfunctions. The unknown coefficients in the series expansion are determined by solving the algebraic system derived by taking successive inner products with each member of the CON set of eigenfunctions. The proposed method and its convergence are demonstrated by solving two model sixth-order BVPs.
\end{abstract}




\section{Introduction}\label{sec:Intro}

Sixth-order parabolic equations arise in a number of physical contexts. For example, early work by King \cite{King1989} on isolation oxidation of silicon led to a model featuring a sixth-order parabolic equation due to the effect of diffusion and reaction on substrate deformation during the oxidation process. Importantly, King \cite{King1989} observed that the resulting partial differential equation is a higher-order version of the well-known nonlinear, degenerate parabolic equations of second and fourth order that describe the height of a thin liquid film during gravity- and surface-tension-driven spreading on a rigid surface. At vastly different scales, the flow of magma under earthen layers leads to the formation of domes called \emph{laccoliths}. The model used the describe the formation and spread of these geological features \cite{Michaut2011,Bunger2011,Bunger2011b,Thorey2014} also led to sixth-order parabolic equations, similar to the one derived by King \cite{King1989}, but usually in axisymmetric polar coordinates. 

In the last two decades, there has been significant interest in fundamental questions about the behavior of solutions of nonlinear sixth-order parabolic equations, motivated by problems of thin liquid films spreading underneath an elastic membrane \cite{Huang2002,Flitton2004,Hosoi2004,Hewitt2015,Pedersen2019,Peng2020}. With the advent of cheap and rapid fabrication methods and improved imaging modalities, a growing area of application for these models and studies has been actuation in microfluidics \cite{Boyko2019,Boyko2020}, including impact mitigation \cite{Tulchinsky2016}. In these contexts, it is often of interest to understand infinitesimal perturbations of the film and their stability, leading to sixth-order eigenvalue problems.

In the mathematics literature, general properties of sixth-order eigenvalue problems have been discussed by Greenberg and Marletta \cite{GM98,GM00}. However, the eigenfunctions are seldom explicitly constructed, the studies focusing instead on calculation of eigenvalues for problems arising in hydrodynamic stability \cite{Chandrasek} or vibrations of sandwich beams \cite{DiTaranto1965,Mead1969}. We surmise that constructing the eigenfunctions of sixth-order eigenvalue problems would be enlightening, and that these functions may resemble the well-known ``beam'' functions \cite{Chandrasek,Chri_Annuary_Chandra,PCB_IJNMF} of the fourth-order eigenvalue problem. 

Thus, in the present work, motivated by recent interest in sixth-order parabolic equations, we study a model initial-boundary-value problem for a sixth-order equation (section~\ref{sec:Target}), construct the eigenfunctions of its associated eigenvalue problem (section~\ref{sec:SLBVP}), use the latter to propose a Galerkin spectral method \cite{Boyd,Shen} for the solution of the former (section~\ref{sec:Galerkin}), and demonstrate the approach on two elementary model problems (section~\ref{sec:model_problems}).

\section{Problem formulation: elastic-plated thin film}\label{sec:Target}

\subsection{Problem definition and flow geometry}\label{sec:Target_definition}

Here, we study a thin fluid film of equilibrium height $h_0$, underneath an elastic interface in a closed trough of width $2\ell$, as shown in figure~\ref{fig:elastic_film_shematic}. The fluid is considered incompressible and Newtonian with (constant) density $\rho_f$ and (constant) dynamic viscosity $\mu_f$. The elastic sheet has  (constant) in-plane tension $T$ and (constant) bending resistance $B$. The sheet is idealized as an interface with no mass (thus, no inertia during its motion). Gravity is the only {body} force considered to act on the fluid, and it acts in the $-y$ direction. The elastic interface is free to move vertically along the lateral walls at $x=\pm\ell$ while maintaining a $90^\circ$ contact angle (non-wetting).

\begin{figure}[ht]
	\begin{center}
		\includegraphics[width=.65\textwidth]{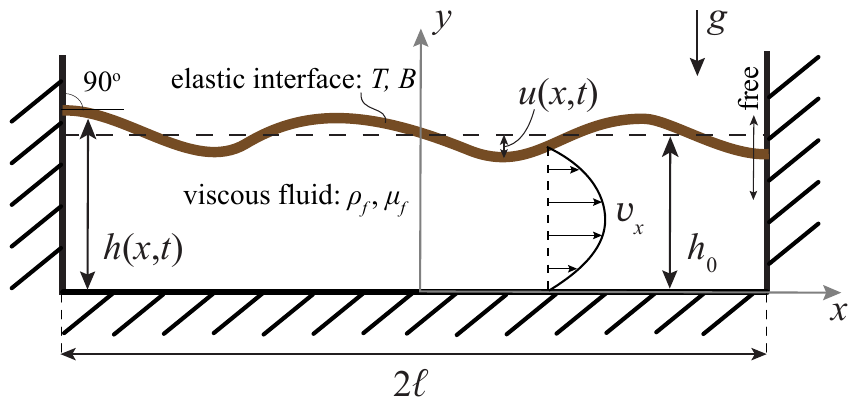}
	\end{center}
	\caption{Schematic of an elastic sheet (with in-plane elastic tension $T$, out-of-plane bending rigidity $B$, and negligible mass), overlaying a viscous fluid (density $\rho_f$ and dynamic viscosity $\mu_f$) of equilibrium height $h_0$ in a closed trough of axial width $2\ell$. A non-uniform height $h(x,t) = h_0 + u(x,t)$ drives a horizontal flow $\varv_x(x,y,t)$ under the elastic interface, and vice versa. Gravity acts in the $-y$-direction, and the gravitational acceleration is denoted by $g$.}
	\label{fig:elastic_film_shematic}
\end{figure}

\subsection{Derivation of the governing thin-film equation}

We study the dynamics of the thin film in the long-wave limit \cite{Oron1997,Craster2009}, also known as the lubrication approximation \cite{Leal2007,Stone2017LH}, such that $h_0\ll\ell$ as well as $\max_{x,t} h(x,t)\ll \ell$ during the film's evolution. Under this assumption, it can be shown \cite{Flitton2004,Hosoi2004,Hewitt2015,Pedersen2019} that the dimensional governing equations for the viscous fluid flow under the elastic sheet reduce to:
\begin{subequations}\begin{alignat}{3}
	&\text{conservation of mass:}\qquad &0 &= \frac{\partial \varv_x}{\partial x} + \frac{\partial \varv_y}{\partial y},\label{eq:continuity}\\
    &x\text{-momentum balance:}\qquad &0 &\approx -\frac{\partial p}{\partial x} + \mu_f \frac{\partial^2 \varv_x}{\partial y^2},\label{eq:elastic_film_xmom}\\
    &y\text{-momentum balance:}\qquad &0 &\approx -\frac{\partial p}{\partial y} - \rho_f g, \label{eq:elastic_film_ymom}\\
    &\text{dynamic interface condition:}\qquad &-p|_{y=h} &\approx T \frac{\partial^2 h}{\partial x^2}- B \frac{\partial^4 h}{\partial x^4} , \label{eq:elastic_film_interface}\\
    &\text{kinematic interface condition:}\qquad &\frac{\partial h}{\partial t} &= \varv_y|_{y=h}. \label{eq:elastic_film_kinematic}
\end{alignat}\label{eq:elastic_film_all}\end{subequations}
See also \cite{Huang2002,Howell2009} for more in-depth discussions of the dynamic boundary condition at a fluid--elastic solid interface and its linearization in the long-wave limit.

Under the lubrication approximation, the flow is approximately unidirectional and $\varv_y \ll \varv_x$. The dominant velocity component can be obtained from \eqref{eq:elastic_film_xmom} as
\begin{equation}
	\varv_x(x,y,t) = -\frac{1}{2\mu_f}\frac{\partial p}{\partial x} y[h(x,t)-y],
	\label{eq:elastic_film_v_x}
\end{equation}
having enforced the no-slip conditions $\varv_x(x,0,t)=\varv_x(x,h,t)=0$ at the bottom of the trough ($y=0$) and along the elastic interface ($y=h$).
Based on \eqref{eq:elastic_film_v_x}, the volumetric flux \emph{per unit width} is
\begin{equation}
	q := \int_0^{h(x,t)} \varv_x(x,y,t) \,\rd y = -\frac{h(x,t)^3}{12\mu_f}\frac{\partial p}{\partial x}.
\label{eq:elastic_film_flux}
\end{equation}

The vertical velocity can be obtained from the conservation of mass equation~\eqref{eq:continuity} using $\varv_x$ from \eqref{eq:elastic_film_v_x} as
\begin{equation}
	\varv_y(x,y,t) = - \frac{\partial}{\partial x}\int_0^{y} \varv_x(x,\tilde{y},t) \,\rd \tilde{y},
	\label{eq:elastic_film_v_y}
\end{equation}
having enforce the no-penetration condition $\varv_y(x,0,t)=0$ at the bottom of the trough ($y=0$). Using \eqref{eq:elastic_film_flux} and \eqref{eq:elastic_film_v_y}, the kinematic condition~\eqref{eq:elastic_film_kinematic} becomes
\begin{equation}
	\frac{\partial h}{\partial t} = - \frac{\partial q}{\partial x}\,.
	\label{eq:elastic_film_continuity}
\end{equation}

Next, fixing the pressure at the interface $y=h(x,t)$ to be the ``ambient pressure" $p_a=\mathrm{const.}$, equations~\eqref{eq:elastic_film_ymom} and \eqref{eq:elastic_film_interface} allow us to solve for the pressure distribution in the thin fluid film:
\begin{equation}
	p(x,y,t) = p_a + \rho_f g[h(x,t)-y] - T \frac{\partial^2 h}{\partial x^2} + B \frac{\partial^4 h}{\partial x^4}.
\label{eq:elastic_film_p}
\end{equation}
Finally, using the pressure~\eqref{eq:elastic_film_p} to eliminate $\partial p/\partial x$ from the flow rate expression~\eqref{eq:elastic_film_flux}, and substituting the result into \eqref{eq:elastic_film_continuity}, we obtain the (long-wave) \emph{thin film equation}:
\begin{equation}
	\frac{\partial h}{\partial t} = \frac{\partial}{\partial x} \left[ \frac{h^3}{12\mu_f}\left( \rho_f g  \frac{\partial h}{\partial x} - T \frac{\partial^3 h}{\partial x^3} + B \frac{\partial^5 h}{\partial x^5} \right)\right].
	\label{eq:elastic_film}
\end{equation}
Observe that \eqref{eq:elastic_film} is a nonlinear partial differential equation (PDE), nominally of parabolic type, but with possible degeneracy at a ``contact line'' where $h\to0$, which leads to a wealth of possible contact-line behaviors \cite{Flitton2004,Lister2019} beyond the scope of this work.

\subsection{Nondimensionalization}

We introduce the dimensionless (``hat'') variables via the transformations:
\begin{equation}
	x = \ell\hat{x}, \qquad t = \frac{12\mu_f \ell^6}{Bh_0^3} \hat{t}, \qquad h = h_0 \hat{h}.
	\label{eq:elastic_film_ndvar}
\end{equation}
We focus on the case in which the dominant force on the interface is the bending resistance, thus the choice of time scale is $t_c = {12\mu_f \ell^6}/{Bh_0^3}$, as in \cite{Tulchinsky2016,MartinezCalvo2020}.
Substituting the variables from \eqref{eq:elastic_film_ndvar} into \eqref{eq:elastic_film} and dropping the hats, we arrive at:
\begin{equation}
	\frac{\partial h}{\partial t} = \frac{\partial}{\partial x} \left[ h^3\left( \mathcal{B} \frac{\partial h}{\partial x} - \mathcal{T} \frac{\partial^3 h}{\partial x^3} + \frac{\partial^5 h}{\partial x^5} \right) \right].
	\label{eq:elastic_film_nd3}
\end{equation}
The following dimensionless numbers have arisen:
\begin{subequations}\begin{alignat}{3}
	&\text{Elastic Bond number:}\qquad &\mathcal{B} &= \frac{\rho_f g \ell^4}{B},\\
	&\text{Tension number:}\qquad &\mathcal{T} &= \frac{T\ell^2}{B}.
	\label{eq:elastic_film_ndnum}
\end{alignat}\end{subequations}
The elastic Bond number $\mathcal{B}$ quantifies the relative importance of gravity's role in deforming the elastic interface relative to its bending resistance (see also discussion in \cite{Duprat2011}). The tension number $\mathcal{T}$ quantifies the relative importance of in-plane tension's role in deforming the elastic interface relative to its bending resistance (see also the discussions in \cite{Howell2009,Anand2020} and \cite{Boyko2019,Boyko2020}).

We are interested in small perturbations to a flat film, $h=1$, so we introduce $h(x,t) = 1 + \varepsilon u(x,t)$ with $\varepsilon\ll1$ into \eqref{eq:elastic_film_nd3} and keep only terms at $\mathrm{O}(\varepsilon)$, to obtain
\begin{equation}
	\frac{\partial u}{\partial t} = \mathcal{B} \frac{\partial^2 u}{\partial x^2} - \mathcal{T}\frac{\partial^4 u}{\partial x^4} + \frac{\partial^6 u}{\partial x^6}.
	\label{eq:elastic_film_linear}
\end{equation}
Equation~\eqref{eq:elastic_film_linear} is the \emph{linear} sixth-order PDE of interest in the remainder of this work.

\subsection{Boundary conditions}

Equation~\eqref{eq:elastic_film_linear} requires six boundary conditions (BCs). Some are discussed in \cite{Hosoi2004} for a thin film problem and in \cite[section~4.4.2]{Howell2009} in the context of beams. According to the problem posed in section~\ref{sec:Target_definition} and shown in figure~\ref{fig:elastic_film_shematic}, we are interested in a nonwetting, non-pinned film in a closed trough. A non-wetting film will maintain a $90^\circ$ contact at the wall, thus
\begin{subequations}
\begin{equation}
	\left. \frac{\partial u}{\partial x}\right|_{x=\pm1} = 0.
	\label{eq:elastic_film_BC1}
\end{equation}
Since the film is not pinned (i.e., it is free), $u(\pm1,t)$ is free and not necessarily $=0$. Similarly, since the film is confined in a trough with walls, the walls provide shear force/resistance to the motion, so that $(\partial^3 u/\partial x^3)|_{x=\pm1}$ is not necessarily $=0$, but this in turn requires that there be no moment at the wall:
\begin{equation}
	\left. \frac{\partial^2 u}{\partial x^2}\right|_{x=\pm1} = 0. \label{eq:elastic_film_BC2}\\
\end{equation}
Finally, we observe from the derivation of \eqref{eq:elastic_film}, that if the fluid flux $q$ is to vanish at $x=\pm1$ (closed trough) then, taking into account \eqref{eq:elastic_film_BC1} and \eqref{eq:elastic_film_BC2}, we must additionally require that 
\begin{equation}
	\left. \left[ -\mathscr{T} \frac{\partial^3 u}{\partial x^3} + \frac{\partial^5 u}{\partial x^5} \right]\right|_{x=\pm1} = 0. \nonumber
\end{equation}
If we restrict ourselves to cases in which tension in negligible compared to bending ($\mathscr{T}\to0$), then
\begin{equation}
	\left. \frac{\partial^5 u}{\partial x^5}\right|_{x=\pm1} = 0. \label{eq:elastic_film_BC3}
\end{equation}\label{eq:elastic_film_BCs}%
\end{subequations}

\section{The sixth-order Sturm--Liouville boundary value problem}\label{sec:SLBVP}

In the context of the PDE~\eqref{eq:elastic_film_linear} subject to the BCs~\eqref{eq:elastic_film_BCs}, we consider the following sixth-order initial boundary value problem (IBVP) for $u=u(x,t)$:
\begin{subequations}
	\begin{empheq} [left = \empheqlbrace]{alignat=2}
		\frac{\partial u}{\partial t} - \mathcal{B} \frac{\partial^2 u}{\partial x^2} + \mathcal{T}\frac{\partial^4 u}{\partial x^4} - \frac{\partial^6 u}{\partial x^6} &= f(x,t), \quad &(x,t)\in(-1,1)\times\in(0,\infty), \label{eq:PDE_6}\\
		\left.\frac{\partial u}{\partial x}\right|_{x=\pm1} = \left.\frac{\partial^2 u}{\partial x^2}\right|_{x=\pm1} = \left.\frac{\partial^5 u}{\partial x^5}\right|_{x=\pm1} &= 0,\quad &t\in(0,\infty), \label{eq:BC_6}\\
		u(x,0) &= u^0(x), \quad &x\in(-1,1),
	\end{empheq}\label{eq:IBVP_6}%
\end{subequations}
where $u^0(x)$ and $f(x,t)$ are given smooth functions. Although, on physical grounds, we took $\mathscr{T}\to0$ to derive BC~\eqref{eq:elastic_film_BC3}, we have kept $\mathscr{T}$ in \eqref{eq:PDE_6} for the purposes of the mathematical discussion that follow.

To solve IBVP~\eqref{eq:IBVP_6}, we will employ a Galerkin method (section~\ref{sec:Galerkin}). To this end, we must construct a suitable set of eigenfunctions, and determine their properties, which is the subject of the remainder of this section.

\subsection{Self-adjointness, orthogonality and completeness}

The Sturm--Liouville eigenvalue problem (EVP) associated with IBVP~\eqref{eq:IBVP_6} reads:
\begin{subequations}
	\label{eq:6th_EVP}
	\begin{align}
		\label{eq:ODE_hg_6} 
		 -\frac{\rd^6 \psi}{\rd x^6} &= \lambda^6 \psi , 	\\
		\left.\frac{\rd \psi}{\rd x}\right|_{x=\pm1} = \left.\frac{\rd^2 \psi}{\rd x^2}\right|_{x=\pm1}&= \left.\frac{\rd^5 \psi}{\rd x^5}\right|_{x=\pm1}=0 .
		\label{eq:6th_BCs}
	\end{align}
\end{subequations}	
Before we proceed with the derivation of the solutions (eigenfunctions) of  EVP~\eqref{eq:6th_EVP}, we first discuss why they possess certain properties required for the development of our Galerkin method in section~\ref{sec:Galerkin}.   
The desired results follow as direct applications of the comprehensive general theory on eigenvalue problems found in the seminal book of Coddington and Levinson~\cite{CodLev}. To facilitate this discussion for the reader, we quote or paraphrase the relevant definitions and theorems from \cite{CodLev} accordingly. We begin with the definitions.
\begin{definition}{[Adjoint of a linear operator (from \cite{CodLev}, chapter~3, section~6)]}
	 Consider the linear differential operator 
	 \begin{equation}
	 	\mathscr{L}_n(\,\cdot\,) := a_0\frac{\rd^n}{\rd x^n}(\,\cdot\,) + a_1\frac{\rd^{n-1}}{\rd x^{n-1}}(\,\cdot\,) + \cdots + (a_n\,\cdot\,),
	 \end{equation}	
	 where $a_0, a_1, \ldots, a_n$ are continuous (possibly complex-valued) functions of $x$ on some interval $[a,b]\subset\mathbb{R}$, with $a_0(x)\neq0$ $\forall x\in [a,b]$. Then, the adjoint of $\mathscr{L}_n $ is another linear operator, denoted $\mathscr{L}^+_n$, also of order $n$ given by
	 \begin{equation}
			\mathscr{L}^+_n(\,\cdot\,) := (-1)^n\frac{\rd^n}{\rd x^n}(\bar{a}_0\,\cdot\,) + (-1)^{n-1}\frac{\rd^{n-1}}{\rd x^{n-1}}(\bar{a}_1\,\cdot\,) + \cdots + (\bar{a}_n\,\cdot\,) ,
      \end{equation}
      where the bar (\, $\bar{\ }$\, ) denotes the complex conjugate.
\end{definition}	
	
	Throughout this work we will use the notation $\langle \cdot,\cdot \rangle$ for the  $L^2$ inner product on the interval $[a,b]$, that is, for any $\phi, \psi\in L^2[a,b]$, 
	\begin{equation}
		\langle \phi, \psi \rangle := \int_{a}^{b} \phi(x) \psi(x) \,\dif x .
	 \end{equation}
	 Also, $\|\cdot\|$ will denote the associated $L^2$ norm unless explicitly specified otherwise. 
	
	 \begin{definition}	{[Self-adjoint eigenvalue problem (from \cite{CodLev}, chapter~7, section~2)]}
	 	
	 	 Let $\mathscr{L}_n$ be the $n$-th order operator given by  
	 	 \begin{subequations}
	 	 	\label{eq:selfadjoint_EVP_gen}
	 	 \begin{equation}
	 	 	\label{eq:nth_order_operator}
	 	 	\mathscr{L}_n \phi = p_0\frac{\rd^n \phi}{\rd x^n} + p_1\frac{\rd^{n-1} \phi}{\rd x^{n-1}} + \cdots + p_n \phi,
	 	 \end{equation}	
	 	 where the $p_i(x)$ are (possibly complex-valued) functions of class $\mathcal{C}^{n-j}$, $j=0,1,\ldots,n$ on the closed interval $[a,b]$ and $p_0(x)\neq0$ on $[a,b]$. Let
	 	 \begin{equation}
	 	 	\label{eq:nth_order_BCs}
	 	 	U_j \phi = \sum_{k=1}^{n} \bigg(M_{jk}\phi^{(k-1)}(a) + N_{jk}\phi^{(k-1)}(b)\bigg) \qquad \qquad (j=1,\ldots,n),
	 	 \end{equation}
	 	 where the $M_{jk}$ and $N_{jk}$ are constants. Denote the boundary conditions $U_j\phi = 0$,  $j=1,\ldots,n$ by $\mathscr{U}\phi=0$. 
	 	 The eigenvalue problem 
	 	 \begin{equation}
	 	      \mathscr{L}_n \phi = \lambda \phi \, ,\qquad \mathscr{U}\phi=0,
	 	\end{equation}
	 	 is \emph{self-adjoint} if  
	 	\begin{equation}
	 		\label{eq:nth_order_selfadjoint}
	 		\langle \mathscr{L}_n \phi, \psi \rangle = \langle \phi, \mathscr{L}_n  \psi \rangle
	 	\end{equation}	
	 	 for all $\phi$, $\psi$ $\in$ $\mathcal{C}^n[a,b]$ which satisfy the boundary conditions $\mathscr{U}\phi=\mathscr{U}\psi=0$. 
	 	  \end{subequations}
	 \end{definition}

The first important consequence of self-adjointness is reflected in theorem~\ref{thm:enum_ortho}.
\begin{theorem}{[From \cite{CodLev}, theorem~2.1 of chapter~7, section~2]}
	\label{thm:enum_ortho}
	A self-adjoint EVP has a real, enumerable set of eigenvalues with corresponding distinct and mutually-orthogonal eigenfunctions.
\end{theorem}

We now show the following proposition.
\begin{proposition}	
	\label{thm:6th_EVP_self_adj}
	The sixth-order EVP~\eqref{eq:6th_EVP} with BCs~\eqref{eq:6th_BCs} is \emph{self-adjoint}. 
\end{proposition}	

\begin{proof}
The operator $\mathscr{L} := -\rd^6/\rd x^6$ from \eqref{eq:ODE_hg_6} is symmetric ($\mathscr{L}=\mathscr{L}^+$).  Let $\psi(x)$ be a solution of EVP~\eqref{eq:6th_EVP} and $\phi(x)\in\mathcal{C}^6[-1,1]$ be any function satisfying BCs~\eqref{eq:6th_BCs}. The $L^2[-1,1]$ inner product between equation~\eqref{eq:ODE_hg_6} and $\phi$ is:
\begin{equation}
	\int_{-1}^{+1} -\frac{\rd^6 \psi}{\rd x^6} \phi \,\rd x = \lambda^6 \int_{-1}^{+1} \psi\phi \,\rd x.
\end{equation}
Integrating the left-hand side by parts yields
\begin{equation}
	- \left.\left[ \frac{\rd^5\psi}{\rd x^5}\phi \right]\right|_{x=-1}^{x=+1}
	+ \int_{-1}^{+1} \frac{\rd^5 \psi}{\rd x^5} \frac{\rd\phi}{\rd x} \,\rd x = \lambda^6 \int_{-1}^{+1} \psi\phi\, 	\rd x.
	\label{eq:i_parts_6_1}
\end{equation}
Repeating five more times, we find
\begin{equation}
	\langle \mathscr{L}\psi,\phi \rangle = J(\psi,\phi) + \langle \psi,\mathscr{L}\phi \rangle,
	\label{eq:Lpp_6}
\end{equation}
where
\begin{equation}
	J(\psi,\phi) := \left.\left[
	\frac{\rd^5\psi}{\rd x^5}\phi
	- \frac{\rd^4 \psi}{\rd x^4}\frac{\rd\phi}{\rd x}
	+ \frac{\rd^3 \psi}{\rd x^3}\frac{\rd^2\phi}{\rd x^2}
	- \frac{\rd^2 \psi}{\rd x^2}\frac{\rd^3\phi}{\rd x^3}
	+ \frac{\rd \psi}{\rd x}\frac{\rd^4\phi}{\rd x^4}
	- \psi\frac{\rd^5\phi}{\rd x^5}
	\right]\right|_{x=-1}^{x=+1} .
	\label{eq:Jpp_6}
\end{equation}

For the BCs~\eqref{eq:6th_BCs}, we find $J(\psi,\phi) =0$. Hence,  
\begin{equation}
	\langle \mathscr{L}\psi,\phi \rangle = \langle \psi,\mathscr{L}\phi \rangle,
	\label{eq:self_adj_6}
\end{equation}
for any solution $\psi(x)$ of EVP~\eqref{eq:6th_EVP} and any function $\phi(x)\in\mathcal{C}^6[-1,1]$ satisfying the BCs~\eqref{eq:6th_BCs}.
\end{proof}

The second important result regarding self-adjoint EVPs was shown in the expansion and completeness theorems, namely, theorem~4.1 and theorem~4.2 of chapter~7, section~4, of \cite{CodLev}. We paraphrase this result as follows.
\begin{proposition}{[Expansion and completeness]}
	\label{thm:Exp_comp_gen}
	Let $\{\chi_m\}_{m=1}^{\infty}$ be an orthonormal sequence of solutions of the self-adjoint EVP~\eqref{eq:selfadjoint_EVP_gen} on $[a,b]$. Then, for any $f\in L^2[a,b]$,  
	\begin{equation}
		f=\sum_{m=1}^{\infty} \langle f,\chi_m\rangle \chi_m,
	\end{equation}	
	where the equality implies convergence in the $L^2$ norm.
\end{proposition}

\subsection{Constructing the orthonormal eigenfunctions}

We now solve EVP~\eqref{eq:6th_EVP} and derive our complete orthonormal (CON) set of solutions. The characteristic equation corresponding to \eqref{eq:ODE_hg_6} has roots
\begin{equation}
	\left\{ +\lambda \ri, -\lambda \ri, \left(\tfrac{\ri}{2} + \tfrac{\sqrt{3}}{2}\right)\lambda, \left(\tfrac{\ri}{2} - \tfrac{\sqrt{3}}{2}\right)\lambda, \left(-\tfrac{\ri}{2} + \tfrac{\sqrt{3}}{2}\right)\lambda, \left(-\tfrac{\ri}{2} - \tfrac{\sqrt{3}}{2}\right)\lambda \right\}.
\end{equation} 
Therefore, the general solution of the homogeneous version of \eqref{eq:ODE_hg_6} can be written in the form
\begin{multline}
	\label{eq:gensol_ODE_hg_6} 
	\psi(x) = \overbrace{c_1 \sin(\lambda x) 
		+ c_5 \sin\left(\tfrac{1}{2} \lambda x\right) \cosh \left(\tfrac{\sqrt{3}}{2} \lambda x\right) 
		+ c_6 \cos\left(\tfrac{1}{2} \lambda x\right) \sinh \left(\tfrac{\sqrt{3}}{2} \lambda x\right)}^{\psi^s(x)\;\;\text{(odd)}}\\
	+ \underbrace{c_2 \cos(\lambda x)
		+ c_3 \sin\left(\tfrac{1}{2} \lambda x\right) \sinh \left(\tfrac{\sqrt{3}}{2} \lambda x\right)
		+ c_4 \cos\left(\tfrac{1}{2} \lambda x\right) \cosh \left(\tfrac{\sqrt{3}}{2} \lambda x\right)}_{\psi^c(x)\;\;\text{(even)}}.
\end{multline}
We have chosen to express the solution in this way, following the example of the so-called ``beam'' functions \cite{Chandrasek,Chri_Annuary_Chandra,PCB_IJNMF}, in order to have odd and even sets of eigenfunctions resembling trigonometric sines and cosines.  

Imposing BCs~\eqref{eq:6th_BCs} on each of the $\psi^s(x)$ and $\psi^c(x)$ separately yields homogeneous linear systems for the  vectors $\bm{c}_{\mathrm{odd}}=(c_1,c_5,c_6)^\top$ and $\bm{c}_{\mathrm{even}}=(c_2,c_3,c_4)^\top$, respectively. To find nontrivial solutions we set the determinants of the coefficient matrices equal to zero and arrive at the eigenvalue relations:
\begin{subequations}\label{eq:evals_6_125}\begin{alignat}{3}
		&\text{even}:\qquad & \cos (2 \lambda^c) + \sqrt{3}\sin\lambda^c \sinh(\sqrt{3}\,\lambda^c) - \cos\lambda^c \cosh(\sqrt{3}\,\lambda^c) &= 0,  \label{eq:evals_e_3_6}\\
		&\text{odd}:\qquad & \sin (2 \lambda^s) + \sqrt{3}\cos\lambda^s \sinh(\sqrt{3}\,\lambda^s) + \sin\lambda^s \cosh(\sqrt{3}\,\lambda^s) &= 0.  \label{eq:evals_o_3_6}
\end{alignat}\end{subequations}

\begin{table}[ht]
		\caption{\label{tab:eigenvalues}``Even'' and ``odd'' eigenvalues, evaluated by solving the transcendental equations~\eqref{eq:evals_6_125} with \textsc{Mathematica}'s  \cite{Mathematica} \texttt{FindRoot}, as well as the proposed asymptotic formulas.} 
	\centering 
	\begin{tabular}{l*{4}{d{12}}}    
		\toprule
		& \multicolumn{2}{c}{even} & \multicolumn{2}{c}{odd}\\ \cmidrule(lr){2-3}  \cmidrule(lr){4-5}\\[-4mm]
		\multicolumn{1}{c}{$m$} & \multicolumn{1}{c}{$\lambda_m^c$} & \multicolumn{1}{c}{$(m+1/6)\pi$} & \multicolumn{1}{c}{$\lambda_m^s$} & \multicolumn{1}{c}{$(m-1/3)\pi$}\\
		\midrule
		\multicolumn{1}{l|}{0} & 0 & - & - & -\\
		\multicolumn{1}{l|}{1} & 3.66606496814 & 3.66519142919 & 2.07175679767 & 2.09439510239\\
		\multicolumn{1}{l|}{2} & 6.80678029161 & 6.80678408278 & 5.23608751229 & 5.23598775598\\
		\multicolumn{1}{l|}{3} & 9.94837675280 & 9.94837673637 & 8.37757997731 & 8.37758040957\\
		\multicolumn{1}{l|}{4} & 13.0899693899 & 13.0899693900 & 11.5191730650 & 11.5191730632\\
		\multicolumn{1}{l|}{5} & 16.2315620436 & 16.2315620435 & 14.6607657167 & 14.6607657168\\
		\multicolumn{1}{l|}{6} & 19.3731546971 & 19.3731546971 & 17.8023583704 & 17.8023583703\\
		\bottomrule
	\end{tabular}
\end{table}

Each solution $\lambda_m^c$ of  \eqref{eq:evals_e_3_6} is an eigenvalue corresponding to an even eigenfunction $\psi^c_m$ of EVP~\eqref{eq:6th_EVP}, whereas each solution $\lambda_m^s$ of  \eqref{eq:evals_o_3_6} corresponds to an odd eigenfunction $\psi^s_m$. 
Eigenvalue relations~\eqref{eq:evals_6_125} were solved numerically using a highly accurate numerical solver---\textsc{Mathematica}'s \texttt{FindRoot} \cite{Mathematica} with 16 digits of working precision. We also derived large-$\lambda$ asymptotic formulas for the eigenvalues. Using \eqref{eq:evals_e_3_6}, we found that $\lambda_m^c \sim (m+1/6)\pi$ as $\lambda\to\infty$, whereas \eqref{eq:evals_o_3_6} yields $\lambda_m^s \sim (m-1/3)\pi$ as $\lambda\to\infty$. The results are presented in table~\ref{tab:eigenvalues}. As we can see, the asymptotic formulas are extremely accurate even for  $m=6$, with the formulas for $\lambda_m^c$ and $\lambda_m^s$ agreeing with the results obtained with \texttt{FindRoot} for 12 and 11 digits, respectively. Thus, in practice when implementing a Galerkin method, there will be no need to use numerical root-finding to determine $\lambda_m^c$ and $\lambda_m^s$ for $m\geq7$.

What remains is to determine the constants $c_2$, $c_3$ and $c_4$ for the set of even functions $\{\psi_m^c\}$ and $c_1$, $c_5$ and $c_6$ for the set  of odd functions $\{\psi_m^s\}$, which were introduced in \eqref{eq:gensol_ODE_hg_6}. These constants are found by imposing BCs~\eqref{eq:6th_BCs}, utilizing the eigenvalue relations~\eqref{eq:evals_6_125}, and normalizing the eigenfunctions with respect to the $L^2$ norm. After a lengthy calculation, using \textsc{Mathematica}'s algebraic manipulation capabilities, we arrive at the expressions: 
\begin{subequations}\label{eq:efuncs_6_125}
	\begin{align}
		\psi_m^c(x) &= c^c_m \Bigg\{
		\tfrac{ 4\sin{\lambda_m^c}}{\cos{\lambda_m^c} - \cosh{(\sqrt{3}\,\lambda_m^c)}} \bigg[ -\cos{\tfrac{\lambda_m^c}{2}} \sinh{\tfrac{\sqrt{3}\,\lambda_m^c}{2}} \sin{\tfrac{\lambda_m^c}{2} x} \sinh{\tfrac{\sqrt{3}\, \lambda_m^c }{2}  x} \nonumber\\ 
		&\qquad\quad +\sin{\tfrac{\lambda_m^c}{2}} \cosh{\tfrac{\sqrt{3}\,\lambda_m^c}{2}} \cos{\tfrac{\lambda_m^c}{2}  x} \cosh{\tfrac{\sqrt{3}\,\lambda_m^c}{2} x} \bigg] + \cos{\lambda_m^c x} \Bigg\},\label{eq:efuncs_6_125_even} \\
		\psi_m^s(x) &= c^s_m \Bigg\{
		\tfrac{ 4\cos{\lambda_m^s}}{\cos{\lambda_m^s} +  \cosh{(\sqrt{3}\,\lambda_m^s)} } \bigg[ -\cos{\tfrac{\lambda_m^s}{2}} \cosh{\tfrac{\sqrt{3}\,\lambda_m^s}{2}} \sin{\tfrac{\lambda_m^s}{2} x} \cosh{\tfrac{\sqrt{3}\,\lambda_m^s}{2}  x}  \nonumber \\		
		&\qquad\quad +\sin{\tfrac{\lambda_m^s}{2}} \sinh{\tfrac{\sqrt{3}\,\lambda_m^s}{2}} \cos{\tfrac{\lambda_m^s}{2}  x} \sinh{ \tfrac{\sqrt{3}\,\lambda_m^s}{2}  x} \bigg] +  \sin(\lambda_m^s x) \Bigg\}, 	\label{eq:efuncs_6_125_odd} 
	\end{align}
where
\begin{align}
	c^c_m &=2 \sqrt{\tfrac{ \lambda^c_m}{d_m^c}} \left[\cos\lambda^c_m - \cosh (\sqrt{3}\,\lambda^c_m)\right], \label{eq:c_m^c_defn}\\
	c^s_m &= 2 \sqrt{\tfrac{\lambda_m^s}{d^s_m}} \left[\cos\lambda_m^s + \cosh (\sqrt{3}\,\lambda_m^s)\right],\label{eq:c_m^s_defn}\\
	d^c_m &= \sin{4\lambda^c_m} - 6\lambda^c_m(\cos{2\lambda^c_m} - 2) + 2 \lambda^c_m \cosh{(2\sqrt{3}\lambda^c_m)} + 2 \sin{2\lambda^c_m} \cosh^2{\sqrt{3}\,\lambda^c_m} \nonumber\\
	&\phantom{=} + \cosh{\sqrt{3}\,\lambda^c_m} \left[\sin{\lambda^c_m} -3 \sin {3\lambda^c_m} + 4\lambda^c_m (\cos{3\lambda^c_m} - 3\cos{\lambda^c_m})\right] \nonumber\\
	&\phantom{=} + 4\sqrt{3} \sin^2{\lambda^c_m} \sinh{\sqrt{3}\,\lambda^c_m} (\cos{\lambda^c_m} - \cosh{\sqrt{3}\lambda^c_m}),\label{eq:d_m^c_defn}\\
	d^s_m &= 12 \lambda^s_m - 3\sin{(2\lambda^s_m)} - \sin{(4\lambda^s_m)} + 10 \lambda^s_m \cos{(2 \lambda^s_m)}\nonumber\\
	&\phantom{=} -[\sin (2\lambda^s_m) - 2\lambda^s_m] \cosh(2 \sqrt{3}\lambda^s_m) - 4\sqrt{3}\cos^2\lambda_m^s \sinh(\sqrt{3} \lambda^s_m) [\cos\lambda^s_m + \cosh(\sqrt{3} \lambda^s_m)] \nonumber\\
	&\phantom{=} +2 \cos\lambda^s_m \cosh(\sqrt{3} \lambda^s_m) \left( 4 \lambda^s_m [\cos(2\lambda^s_m) + 2] - 3 \sin(2 \lambda^s_m) \right), \label{eq:d_m^s_defn}
\end{align}%
\end{subequations}
for $n,m\in\mathbb{N}$.

It is important to note here that $\lambda_0^c=0$ is an eigenvalue of EVP~\eqref{eq:6th_EVP} with corresponding eigenfunction $\psi_0^c(x)=1$. The notation is due to the fact that $\psi(x)=1$ is an even function.  The profiles of the first members of the two sets of eigenfunctions are presented in figure~\ref{fig:efunc_125}.

\begin{figure}[ht]
	\centering
	\begin{subfigure}[b]{0.47\textwidth}
		\centering
		\includegraphics[width=\textwidth]{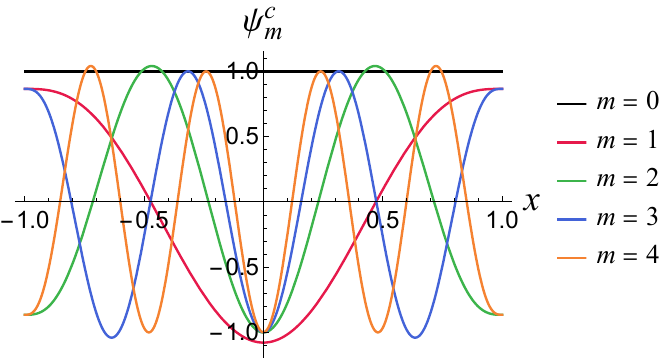}
		\caption{Even eigenfunctions.~\label{fig:efunc_125_even}}
	\end{subfigure}
	\hfill
	\begin{subfigure}[b]{0.47\textwidth}
		\centering
		\includegraphics[width=\textwidth]{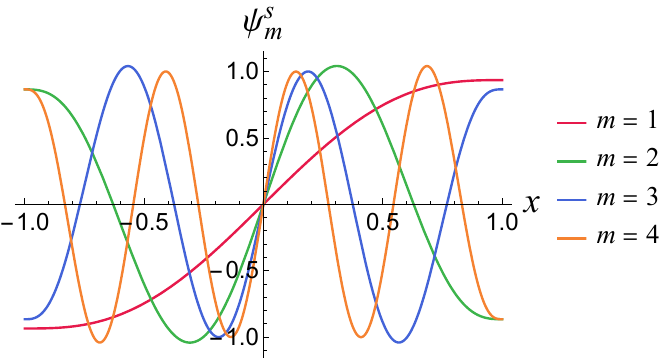}
		\caption{Odd eigenfunctions.~\label{fig:efunc_125_odd}}
	\end{subfigure}
	\caption{The profiles of the (a) even and (b) odd eigenfunctions for $m=1,2,3,4$, along with them the eigenfunction $\psi_0^c(x)=1$ for $m=0$.~\label{fig:efunc_125}}
\end{figure}

We conclude the this section by stating the following proposition, which forms the foundation of our spectral method.
	\begin{proposition}{[Complete orthonormal set]}
		\label{prop:EVP_beam_6}
		The two sets of solutions $\psi_m^c$ and $\psi_m^s$ of  Sturm--Liouville problem~\eqref{eq:6th_EVP} given by \eqref{eq:efuncs_6_125} and supplemented by $\psi_0^c(x) = 1$ form a \emph{complete orthonormal set of functions} (CON) on $L^2[-1,1]$. Therefore, any function $u(x)\in L^2[-1,1]$ can be expanded as 
		\begin{subequations}
			\label{eq:full_expansion_formula} 
			\begin{alignat}{2}
				u(x) &= \frac{1}{2} u_0^c \psi_0^c(x) + \sum_{m=1}^\infty u_m^c \psi_m^c(x) + u_m^s \psi_m^s(x), \quad & \label{eq:u_expansion_formula}\\
				u_0^c &:= \big\langle u(x), \psi_0^c(x) \big\rangle = \int_{-1}^{+1} u(x) \, \rd x, \quad &\\
				u_m^c &:= \left\langle u(x), \psi_m^c(x) \right\rangle = \int_{-1}^{+1} u(x) \psi_m^c(x) \,\rd x,\quad &m\ge 1,\\
				u_m^s &:= \left\langle u(x), \psi_m^s(x) \right\rangle = \int_{-1}^{+1} u(x) \psi_m^s(x) \,\rd x,\quad &m\ge 1,
			\end{alignat}
		\end{subequations}
		with the series \eqref{eq:full_expansion_formula} converging to $u(x)$ in $L^2[-1,1]$.
	\end{proposition}
The proposition follows from proposition~\ref{thm:6th_EVP_self_adj}, theorem~\ref{thm:enum_ortho} and proposition~\ref{thm:Exp_comp_gen}.
It is clear that the only nonzero coefficients in the spectral expansion~\eqref{eq:full_expansion_formula} of any even (or odd) function $f\in L^2[-1,1]$ are the $u_m^c$ (or $u_m^s$). 

In this section $n$ indicated the order of the EVP, and $m$ was the index used for the countable infinity of eigenvalues and eigenfunctions, we have now restricted ourselves to a specific sixth-order EVP. Thus,  in the following sections, we reuse $n$ as the index for the countable infinity of eigenvalues and eigenfunctions to simplify the notation.

\section{The Galerkin spectral method}\label{sec:Galerkin}

\subsection{Method formulation}\label{subsec:MethForm}
Having established the necessary theoretical foundation and derived our CON basis functions, we proceed with the development of the proposed Galerkin spectral method \cite{Boyd,Shen} for IBVP~\eqref{eq:IBVP_6}. To solve the sixth-order PDE~\eqref{eq:PDE_6} using a Galerkin method based on the results above, we expand the sought function (i.e., the solution of the IBVP) $u(x,t)$ as in \eqref{eq:full_expansion_formula}, allowing the expansion coefficients to be functions of time, and truncate the series at $M$ terms:
\begin{equation}
	u(x,t) \approx u_{\mathrm{spectral}}(x,t) = \frac{1}{2} u_0^c(t) \psi_0^c(x) + \sum_{n=1}^{M} u_n^c(t) \psi_n^c(x) + u_n^s(t) \psi_n^s(x).	
	\label{eq:u_x_t_expansion_formula}
\end{equation}
Note that the spectral expansion~\eqref{eq:u_x_t_expansion_formula} will intrinsically satisfy BCs~\eqref{eq:BC_6}, which is an important advantage of  the classical Galerkin approach.

Next, we must express all the $x$-derivatives that appear in the PDE~\eqref{eq:PDE_6} as linear combinations of the same basis functions:
\begin{subequations}\label{eq:deriv_expansion_series}
	\begin{alignat}{2}
		\frac{\rd^2\psi_n^{c,s}}{\rd x^2}  &= \frac{1}{2}\beta_{n0}^{c,s} + \sum_{m=1}^\infty \beta_{nm}^{c,s} \psi_{m}^{c,s}(x),\qquad\text{where}\quad &\beta_{nm}^{c,s} &:= \left\langle \frac{\rd^2\psi_n^{c,s}}{\rd x^2}  ,\psi_m^{c,s} \right\rangle = \int_{-1}^{+1} \frac{\rd^2\psi_n^{c,s}}{\rd x^2} \psi_m^{c,s} \,\rd x,  \label{eq:beta_series}\\
		\frac{\rd^4\psi_n^{c,s}}{\rd x^4}  &= \frac{1}{2}\gamma_{n0}^{c,s} + \sum_{m=1}^\infty \gamma_{nm}^{c,s} \psi_{m}^{c,s}(x),\qquad\text{where}\quad &\gamma_{nm}^{c,s} &:= \left\langle \frac{\rd^4\psi_n^{c,s}}{\rd x^4}  ,\psi_m^{c,s} \right\rangle = \int_{-1}^{+1} \frac{\rd^4\psi_n^{c,s}}{\rd x^4} \psi_m^{c,s} \,\rd x. \label{eq:gamma_series}
	\end{alignat}%
\end{subequations}
Note that, while $\beta_{n0}^{s}=\gamma_{n0}^s=0$ by definition, it can be shown that $\beta_{n0}^c=0$ but $\gamma_{n0}^c\ne0$. The expressions for $\beta_{nm}^{c,s}$ and $\gamma_{nm}^{c,s}$ will be given in section~\ref{sec:expansion_formulas} below. It follows that
\begin{subequations}\label{eq:deriv_spectral_series}
	\begin{align}
		\frac{\partial^2 u_{\mathrm{spectral}}}{\partial x^2}  &=  \sum_{n=1}^M \sum_{m=1}^M u_n^c(t) \beta_{nm}^{c} \psi_{m}^{c}(x) + \sum_{n=1}^M \sum_{m=1}^M u_n^s(t) \beta_{nm}^{s} \psi_{m}^{s}(x),\\
		\frac{\partial^4 u_{\mathrm{spectral}}}{\partial x^4}  &=  \sum_{n=1}^M u_n^c(t) \left[ \frac{1}{2}\gamma_{n0}^{c} + \sum_{m=1}^M \gamma_{nm}^{c} \psi_{m}^{c}(x) \right] + \sum_{n=1}^M \sum_{m=1}^M u_n^s(t) \gamma_{nm}^{s} \psi_{m}^{s}(x),\\
		\frac{\partial^6 u_{\mathrm{spectral}}}{\partial x^6}  &=  -\sum_{n=1}^M (\lambda^{c}_n)^6 u_n^c(t) \psi_n^c(x) - \sum_{n=1}^M (\lambda^{s}_n)^6 u_n^s(t) \psi_{n}^{s}(x).
	\end{align}%
\end{subequations}

Next, substituting \eqref{eq:u_x_t_expansion_formula} and \eqref{eq:deriv_spectral_series} into \eqref{eq:PDE_6}, taking successive inner products with $\psi_0^c(x)=1$, $\big\{\psi_{\ell}^c(x)\big\}_{\ell=1}^{M}$ and $\big\{\psi_{\ell}^s(x)\big\}_{\ell=1}^{M}$, and using the orthogonality of the eigenfunctions, we obtain a semi-discrete (dynamical) system:
\begin{subequations}\label{eq:dynamical_system}
	\begin{alignat}{2}
		\frac{\rd u_0^c}{\rd t} &= \sum_{n=1}^M \left[-\mathcal{T}\gamma_{n0}^{c}\right]u_n^c(t) + f_0^c, \quad &(\ell=0),\\
		\frac{\rd u_{\ell}^c}{\rd t} &=  \sum_{n=1}^M \left[ - \mathcal{T}\gamma_{n \ell}^{c} + \mathcal{B}\,\beta_{n \ell}^{c} - (\lambda^{c}_n)^6 \delta_{nl} \right] u_n^c(t) + f_\ell^c, \qquad &\ell\ge 1,\\
		\frac{\rd u_{\ell}^s}{\rd t} &= \sum_{n=1}^M \left[-\mathcal{T}\gamma_{n \ell}^{s} + \mathcal{B}\,\beta_{n \ell}^{s} - (\lambda^{s}_n)^6 \delta_{nl} \right] u_n^s(t) + f_\ell^s, \qquad &\ell\ge 1.
	\end{alignat}
\end{subequations}
where $\{f_0^c(t),f_m^c(t),f_m^s(t)\}$ are the expansion coefficients of the known right-hand side $f(x,t)$ of \eqref{eq:PDE_6}, according to proposition~\ref{prop:EVP_beam_6}. System~\eqref{eq:dynamical_system} can be solved using a Crank--Nicolson-type approach for the time discretization, but its solution is the object of future work.

\subsection{Expansion formulas}\label{sec:expansion_formulas}

We first present the formula for expanding the second derivative of an even eigenfunction into a series of the even basis functions. After a lengthy calculation, using \textsc{Mathematica}'s algebraic manipulation capabilities, from \eqref{eq:beta_series} we find that for a fixed $n\in\mathbb{N}$, and any $m=1,2,3,\ldots$
\begin{equation}
	\label{eq:sec_deriv_even_form}
	\beta_{nm}^c =
	\begin{cases} 
		\frac{6 c_n^c c_m^c (\lambda _m^c)^3 (\lambda _n^c)^3}{(\lambda _m^c)^6 - (\lambda _n^c)^6}\left[ \frac{\lambda _m^c \sin{\lambda _n^c} \left(\cos{2\lambda _m^c}-\sqrt{3} \sin{\lambda _m^c} \sinh{\sqrt{3}\,\lambda _m^c} - \cos{\lambda _m^c} \cosh{\sqrt{3}\lambda _m^c}\right)}{\cos{\lambda _m^c}-\cosh{\sqrt{3} \lambda _m^c}}\right. &\\
		\qquad \qquad \qquad \qquad \left. +\frac{\lambda _n^c \sin{\lambda _m^c} \left(-\cos{2\lambda _n^c}+\sqrt{3} \sin{ \lambda _n^c} \sinh{\sqrt{3}\,\lambda _n^c}+\cos{\lambda _n^c} \cosh{\sqrt{3} \lambda _n^c}\right)}{\cos{\lambda _n^c}-\cosh{\sqrt{3} \lambda _n^c}}\right]\,\ \hfill \text{for} \  n\neq m,&\\[7mm]
		- \frac{\lambda _n^c {(c_n^c)}^2}{{{8 (\cos \lambda _n^c - \cosh \sqrt{3} \lambda _n^c)^2}}}    \bigg[\lambda _n^c \left(3 \cos{2\lambda _n^c} + \sinh^2{\sqrt{3}\lambda _n^c}+\cosh^2{\sqrt{3}\,\lambda _n^c} \right. &\\
		\left. \qquad \qquad \qquad  \qquad \qquad \qquad \quad  \,+ 4 \sqrt{3} \sin^3{\lambda _n^c} \sinh{\sqrt{3}\,\lambda _n^c} - 4 \cos^3{\lambda _n^c}\cosh{\sqrt{3}\,\lambda _n^c}\right)  & \\
		\qquad \qquad \ +\,2 \sin{\lambda _n^c} (\cos{\lambda _n^c}-\cosh{\sqrt{3} \lambda _n^c}) &\\
		\qquad \qquad \qquad \times \left(\sqrt{3} \sin{\lambda _n^c} \sinh{\sqrt{3}\lambda _n^c}+\cos{\lambda _n^c} \cosh{\sqrt{3} \lambda _n^c} -\cos{2 \lambda _n^c}\right)\bigg]\,\ \hfill \text{for} \  n= m,&
	\end{cases}
\end{equation}
where $c_n^c$ and $c_m^c$ are given by \eqref{eq:c_m^c_defn}.
 
\begin{figure}[h]
	\centering
	\begin{subfigure}[b]{0.47\textwidth}
		\centering
		\includegraphics[width=\textwidth]{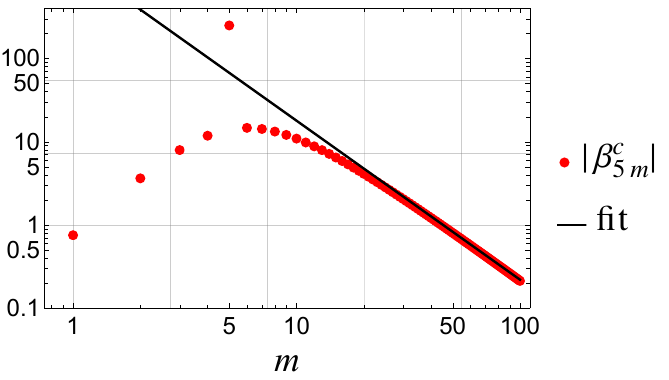}
		\caption{\label{fig:even_beta5m_coeff}}
	\end{subfigure}
	\hfill
	\begin{subfigure}[b]{0.47\textwidth}
		\centering
		\includegraphics[width=\textwidth]{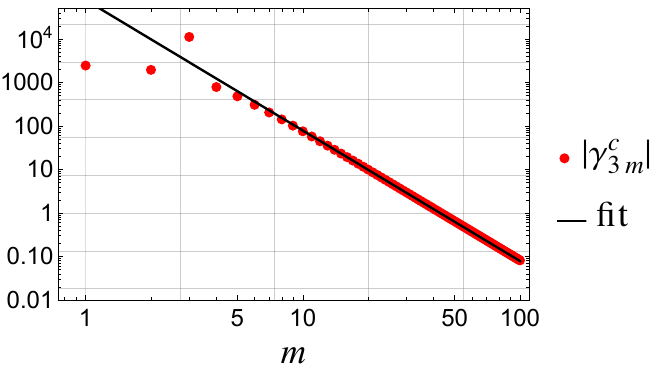}
		\caption{\label{fig:even_gamma3m_coeff}}
	\end{subfigure}
	\caption{Log-log plots demonstrating the convergence rate of the spectral expansions of  (a) the second derivative of the fifth even eigenfunction and (b) the fourth derivative of the third even eigenfunction. Symbols correspond to (a) $|\beta_{5m}^{c}|$ and (b) $|\gamma_{3m}^{c}|$. Solid lines are the best fits based on  values for $m\ge50$; specifically: (a) $1\,455\,m^{-1.91}$ and (b) $78\,160\,m^{-3}$. The series are truncated at $m=M=100$.~\label{fig:deriv_coeff}}
\end{figure} 
The `$c$' case of the expansion~\eqref{eq:beta_series} and \eqref{eq:sec_deriv_even_form} for the second derivative of even eigenfunctions was verified numerically for different values of $n$. In figure~\ref{fig:even_beta5m_coeff} the convergence rate of the spectral series~\eqref{eq:beta_series} for $n=5$ is demonstrated. We observe that $|\beta_{5m}^c|$ decays algebraically for large $m$ as $|\beta_{5m}^c|\sim m^{-1.91}$. This decay was confirmed for all the examined values of $n$.
 
This convergence rate could be very roughly estimated by inspecting  the derived expression for $\beta_{nm}$ and noting that, in this context, $m$ is increasing with $n$ being fixed. Looking at the upper branch of \eqref{eq:sec_deriv_even_form}, i.e., the case $n\neq m$,  and noting that the dominant terms as $m\rightarrow\infty$ are the powers of $\lambda_m^c$,  the hyperbolic functions of  $\lambda_m^c$ and the term $c_m^c$, 
we find $|\beta_{nm}^c| \sim c_m^c \cdot m^{-2}$ for $m$ large. A similar analysis of the lower branch of \eqref{eq:sec_deriv_even_form} explains why terms of the form $\beta_{nn}$ do not follow the overall convergence rate. This observation is confirmed for $\beta_{55}^c$ in figure~\ref{fig:even_beta5m_coeff}.  

For completeness, we also present the corresponding formula for expanding the second derivative of an odd eigenfunction into a series of the odd basis functions. From \eqref{eq:beta_series}, we find that
\begin{equation}
	\label{eq:sec_deriv_odd_form}
	\beta_{nm}^s =
	\begin{cases} 
		\frac{6 c_n^s c_m^s (\lambda _m^s)^3 (\lambda _n^s)^3}{(\lambda _m^s)^6 - (\lambda _n^s)^6} \left[ - \frac{\lambda _m^s \cos{\lambda _n^s} \left(\sin{2\lambda _m^s}-\sqrt{3} \cos{\lambda _m^s} \sinh{\sqrt{3}\,\lambda _m^s} + \sin{\lambda _m^s} \cosh{\sqrt{3}\lambda _m^s}\right)}{\cos{\lambda _m^s}+\cosh{\sqrt{3} \lambda _m^s}}\right. &\\
		\qquad \qquad \qquad \qquad \left. +\frac{\lambda _n^s \cos{\lambda _m^s} \left(\sin{2\lambda _n^s}-\sqrt{3} \cos{ \lambda _n^s} \sinh{\sqrt{3}\,\lambda _n^s}+\cos{\lambda _n^s} \cosh{\sqrt{3} \lambda _n^c}\right)}{\cos{\lambda _n^c}-\cosh{\sqrt{3} \lambda _n^c}}\right]\,\hfill \text{for} \  n\neq m,&\\[7mm]
		 \frac{\lambda _n^s {(c_n^s)}^2}{{{2 (\cos \lambda _n^s + \cosh \sqrt{3} \lambda _n^s )^2}}}     \bigg[\lambda _n^s \left(- \cos{2\lambda _n^s} + \sinh^2{\sqrt{3}\lambda _n^s}+\cosh^2{\sqrt{3}\,\lambda _n^s} \right. &\\
		\left. \qquad  \qquad \qquad \qquad \quad  \,- 4 \sqrt{3} \cos^2{\lambda _n^s}\sin{\lambda_n^s} \sinh{\sqrt{3}\,\lambda _n^s} + 4 \sin^2{\lambda _n^s}\cos{\lambda_n^s}\cosh{\sqrt{3}\,\lambda _n^s}\right)  & \\
		\qquad \qquad \ +\,2 \cos{\lambda _n^s} (\cos{\lambda _n^s}+\cosh{\sqrt{3} \lambda _n^s}) &\\
		\qquad \qquad \qquad \times \left(\sqrt{3} \cos{\lambda _n^s} \sinh{\sqrt{3}\lambda _n^s} - \sin{\lambda _n^s} \cosh{\sqrt{3} \lambda _n^s}-\sin{2 \lambda _n^s}\right)\bigg]\,\hfill \text{for} \  n= m,&
	\end{cases}
\end{equation}
where $c_n^s$ and $c_m^s$ are given by \eqref{eq:c_m^s_defn}. 
The `$s$' case of the expansion~\eqref{eq:beta_series} and \eqref{eq:sec_deriv_odd_form} for the second derivative of odd eigenfunctions was also verified numerically for different values of $n$. Although, we do not show it in a separate plot, we checked the convergence rate of the spectral series~\eqref{eq:beta_series} and \eqref{eq:sec_deriv_odd_form} for $n=5$, observing that $|\beta_{5m}^s|\sim m^{-1.94}$. This decay was confirmed for all the examined values of $n$, and can again be justified by examining the derived expression~\eqref{eq:sec_deriv_odd_form} for $\beta_{nm}^s$.

We also derived the corresponding expressions for expanding the fourth derivatives of our basis functions (see \eqref{eq:gamma_series}). We first present the formulas for the even case. For a fixed $n\in\mathbb{N}$ and $m=1,2,3,\ldots$, we found
\begin{subequations}
		\label{eq:fourth_deriv_even_form_all}
\begin{equation}
	\label{eq:fourth_deriv_even_form}
	\gamma_{nm}^c =
	\begin{cases} 
		\frac{3 c_n^c c_m^c (\lambda _n^c)^6}{(\lambda _m^c)^6 - (\lambda _n^c)^6} \left[ -\frac{(\lambda _m^c)^3 \sin{\lambda _m^c} \left(-3 + \cos{2\lambda _n^c} + 2\cos{\lambda _n^c}\cosh{\sqrt{3}\lambda _n^c}\right)}{\cos{\lambda _n^c} - \cosh{\sqrt{3} \lambda _n^c}}\right. &\\
		\qquad \qquad \qquad \qquad \left. + \frac{(\lambda _n^c)^3 \sin{\lambda _n^c} \left(-3 + \cos{2\lambda _m^c} + 2\cos{\lambda _m^c}\cosh{\sqrt{3}\lambda _m^c}\right)}{\cos{\lambda _m^c} - \cosh{\sqrt{3} \lambda _m^c}} \ \right]\,\hfill \text{for} \  n\neq m,&\\[7mm]
		\frac{(\lambda _n^c)^3 {(c_n^c)}^2}{{{8 \left(\cos \lambda _n^c -\cosh{\!\sqrt{3}\lambda _n^c} \right)^2}}}     \bigg[ 6 \sin{2 \lambda _n^c} \left(\cosh{2\sqrt{3}\lambda _n^c} + 4\right) -3\sin{4\lambda_n^c} - 48\sin{\lambda_n^c}\cosh{\!\sqrt{3}\,\lambda _n^c}&\\
		\qquad   \ + 4\,\lambda _n^c \Big( - 4\sqrt{3}\sin^3{\!\lambda_n^c} \sinh{\!\!\sqrt{3}\lambda _n^c} -  4\cos^3{\!\lambda_n^c} \cosh{\!\!\sqrt{3}\lambda _n^c}&\\
		\qquad \qquad \qquad \qquad	+ 3\cos{2\lambda _n^c} +\cosh{2\!\sqrt{3}\,\lambda _n^c}\Big) \bigg]\,\ \hfill  \text{for} \  n= m.&
	\end{cases}
\end{equation}
As already mentioned in subsection~\ref{subsec:MethForm}, the coefficients $\gamma_{n0}^c\neq0$ and are given by
\begin{flalign}
		\label{eq:fourth_deriv_even_zero_form}
 \quad	\ \ \gamma_{n0}^c = 6 c_n^c (\lambda_n^c)^3 \sin{\lambda_n^c} \ \ \qquad  \qquad \ \ \qquad  \qquad \ \ \qquad  \qquad \ \ \qquad  \qquad   \quad \ \ \ \ \ \text{for} \ n\in\mathbb{N}.&&
\end{flalign}	
\end{subequations}
In figure~\ref{fig:even_gamma3m_coeff} the convergence rate of the spectral series~\eqref{eq:gamma_series} and \eqref{eq:fourth_deriv_even_form_all} for $n=3$ is demonstrated, observing that $|\gamma_{3m}^c|\sim m^{-3}$. This decay was confirmed for all the examined values of $n$, and can again be justified by examining the derived expression~\eqref{eq:fourth_deriv_even_form} for $\gamma_{nm}^s$.


The corresponding expression for the fourth derivative of the odd eigenfunctions reads  
\begin{equation}
	\label{eq:fourth_deriv_odd_form}
	\gamma_{nm}^s =
	\begin{cases} 
		\frac{6 c_n^s c_m^s (\lambda _n^s)^6}{(\lambda _n^s)^6 - (\lambda _m^s)^6} \left[ \frac{(\lambda _m^s)^3 \cos{\lambda _m^s} \sin{\lambda _n^s}\left(\cosh{\sqrt{3}\lambda _n^s} - \cos{\lambda _n^s}\right)}{\cos{\lambda _n^s}+\cosh{\sqrt{3} \lambda _n^s}}\right. &\\
		\qquad \qquad \qquad \qquad \left. + \frac{(\lambda _n^s)^3 \sin{\lambda _m^s}\cos{\lambda _n^s}\left(\cos{\lambda _m^s} - \cosh{\sqrt{3}\lambda _m^s}\right)}{\cos{\lambda _m^s}+\cosh{\sqrt{3} \lambda _m^s}} \ \right]\ \quad \hfill \text{for} \  n\neq m,&\\[7mm]
		\frac{(\lambda _n^s)^3 {(c_n^s)}^2}{{{8 (\cos \lambda _n^s + \cosh \sqrt{3} \lambda _n^s )^2}}}     \bigg[ 6 \sin{2 \lambda _n^s} (\cos{2 \lambda _n^s}-\cosh{2\sqrt{3}\lambda _n^s}) &\\
		\qquad \qquad \qquad  \  + \, 4 \lambda _n^s  \times  \bigg(-\cos{2\lambda _n^s} +\cosh^2{\sqrt{3}\,\lambda _n^s} + \sinh^2{\sqrt{3}\lambda _n^s} &\\
		\qquad \qquad \qquad \quad	+  2\sin{2\lambda _n^s} \left(\sin{\lambda _n^s}\cosh{\sqrt{3}\,\lambda _n^s} + \sqrt{3} \cos{\lambda _n^s}\sinh{\sqrt{3}\,\lambda _n^s} \right) \bigg) \bigg]\,\ \quad  \hfill  \text{for} \  n= m.&
	\end{cases}
\end{equation}
The decay rate of the odd coefficients (not plotted here) is $|\gamma_{3m}^s|\sim m^{-3}$, matching that of its even counterparts. This was confirmed for all the examined values of $n$ and is expected from the expression~\eqref{eq:fourth_deriv_odd_form} for $\gamma_{nm}^s$.

For the model problems considered in section~\ref{sec:model_problems}, the following expansion formulas for even powers of $x$ are also required: 
\begin{equation}
	\label{eq:even_powers_x_expansion}
		x^p = \sum_{m=0}^\infty \chi^{\{p\}}_m \psi_{m}^c(x), \qquad\quad\ \ \chi^{\{p\}}_m = \Big\langle  x^p,\psi_m^c(x)\Big\rangle ,
\end{equation}
where $p=2,4,\ldots,12$. The expressions for the coefficients $\chi^{\{p\}}_m$ for $m\ge1$ are given in the Appendix. The case $m=0$ is not needed, as will become apparent through \eqref{eq:model_1_u0} and 
\eqref{eq:model_2_u0} below.

\section{Model problems: Results and discussion}\label{sec:model_problems}

In this first foray into Galerkin methods for sixth-order-in-space parabolic equations, and the construction of the associated higher-order beam eigenfunctions, we would like to restrict to the steady case, such that $\partial u/\partial t = 0$. In this context, we would like to demonstrate the spectral accuracy of the eigenfunction expansion on two simple \emph{model} boundary-value problems (BVPs), which have simple exact solutions.

\subsection{Model problem I}

Model problem I is the BVP:
\begin{subequations}\label{eq:model_even_1}
    \begin{align}\label{eq:model_even_1:de}
    	\frac{\rd^6 u}{\rd x^6} + 14\, 400\, u &= f(x) ,\\
		\left.\frac{\rd u}{\rd x}\right|_{x=\pm1} = \left.\frac{\rd^2 u}{\rd x^2}\right|_{x=\pm1} = \left.\frac{\rd^5 u}{\rd x^5}\right|_{x=\pm1} &= 0 , \label{model_even_1:bc}
	\end{align}
where
	\begin{equation}
    	f(x) = 216\,000 x^2 - 691\,200 x^4  + 377\,280 x^6  + 216\,000 x^8 - 86\,400 x^{10} + 14\,400 x^{12}.
		\label{eq:model_even_1:rhs}
	\end{equation}%
\end{subequations}
The problem~\eqref{eq:model_even_1} admits the exact solution:
\begin{equation}
	u_\mathrm{exact}(x) = (x-1)^6(x+1)^6.
	\label{eq:model_even_1:exact}
\end{equation}
The exact solution~\eqref{eq:model_even_1:exact} as well as $f(x)$ in \eqref{eq:model_even_1:de} are even functions, and BCs~\eqref{model_even_1:bc} are symmetric. Thus, we need \emph{only} use the \emph{even} eigenfunctions $\psi_0^c(x)=1$ and $\big\{\psi_n^c(x)\big\}_{n=1}^{\infty}$ for the spectral expansion.		
		
Applying the Galerkin spectral method introduced in section~\ref{sec:Galerkin} leads us to introduce the truncated series 
\begin{equation}
	u(x) \approx u_{\mathrm{spectral}}(x) = \frac{1}{2}u_0^c \psi_0^c(x) + \sum_{n=1}^{M} u_n^c \psi_n^c(x)\,.
	\label{eq:u_x_expansion_formula}
\end{equation}
Now, we substitute \eqref{eq:u_x_expansion_formula} into \eqref{eq:model_even_1:de}, take inner products with $\psi_0^c(x)=1$ and $\big\{\psi_{m}^c(x)\big\}_{m=1}^{M}$, and employ the orthogonality relations, as in the derivation of \eqref{eq:dynamical_system}. First, the coefficient $u_0^c$ is found as
\begin{equation}
	14\,400 u_0^c = \int_{-1}^{+1} f(x) \psi_0(x) \,\rd x \quad \Rightarrow \quad u_0^c = \frac{2\,048}{3\,003}\,.
	\label{eq:model_1_u0}
\end{equation}
Next, the remaining coefficients $\{u_n^c\}_{n=1}^M$ are obtained as the solution of the \emph{diagonal} system of equations  
\begin{multline}
	\big[14\,400-(\lambda^c_{n})^6\big] \delta_{nm} u_{n}^c = 216\,000 \chi^{\{2\}}_{m} - 691\,200 \chi^{\{4\}}_{m} + 377\,280 \chi^{\{6\}}_{m}  \\ 
	+ 216\,000 \chi^{\{8\}}_{m} - 86\,400 \chi^{\{10\}}_{m} + 14\,400 \chi^{\{12\}}_{m},\qquad 	
	m=1,2,\ldots,M,
	\label{eq:model_1_um}
\end{multline}
where $\delta_{nm}$ is Kronecker's delta.

The exact solution~\eqref{eq:model_even_1:exact} and the spectral approximation~\eqref{eq:u_x_expansion_formula}--\eqref{eq:model_1_um} using the sixth-order eigenfunctions are compared in figure~\ref{fig:model_even_1_shape}, showing they agree identically (at least visually). As seen in figure~\ref{fig:model_even_1_error}, the spectral expansion is \emph{highly} accurate with the maximum absolute error being $\le 5\times 10^{-13}$. Despite the Gibbs effect near $x=\pm1$, figure~\ref{fig:model_even_1_coeff_decay} shows that the overall convergence rate of the spectral expansion is $\mathrm{O}(n^{-8})$. This extremely rapid convergence is due to the very definition of our eigenfunctions, namely that they satisfy a sixth-order BVP. Put differently: when solving the algebraic system~\eqref{eq:model_1_um}, we are essentially dividing with quantities involving $(\lambda_n^c)^6$, and it is thanks to this fact that we have the impressive convergence rate, overcoming the more moderate decay rate of the coefficients appearing in the right-hide side of  \eqref{eq:model_1_um}.

\begin{figure}
	\centering
	\begin{subfigure}[b]{0.4\textwidth}
		\centering
		\includegraphics[width=\textwidth]{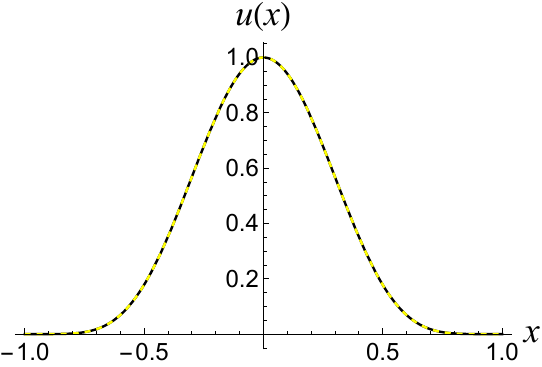}
		\caption{\label{fig:model_even_1_shape}}
	\end{subfigure}
	\qquad
	\begin{subfigure}[b]{0.4\textwidth}
		\centering
		\includegraphics[width=\textwidth]{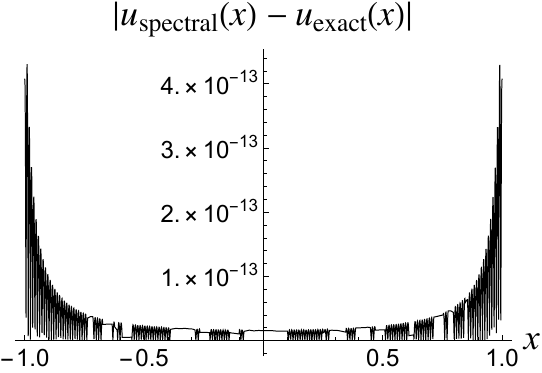}
		\caption{\label{fig:model_even_1_error}}
	\end{subfigure}
	\\
	\begin{subfigure}[b]{0.47\textwidth}
		\centering
		\includegraphics[width=\textwidth]{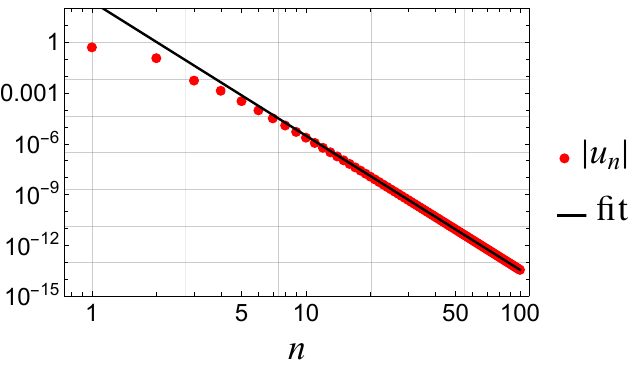}
		\caption{\label{fig:model_even_1_coeff_decay}}
	\end{subfigure}
	\caption{(a) Comparison of the exact~\eqref{eq:model_even_1:exact} (solid) and spectral expansion~\eqref{eq:u_x_expansion_formula}--\eqref{eq:model_1_um} (dashed) solution profiles for model problem I with $M=100$. (b) Absolute error between the exact solution and the spectral expansion. (c) Convergence rate of the spectral expansion with symbols showing $|u_n^c|$ and the solid line being the best fit $259 n^{-7.94}$  based on values for $n\ge50$.~\label{fig:model_even_1}}
\end{figure}

\subsection{Model problem II}

Model problem II is the BVP:
\begin{subequations}\label{eq:model_even_2}
    \begin{align}\label{eq:model_even_2:de}
		\frac{\rd^6 u}{\rd x^6}  -5544 \frac{\rd^2 u}{\rd x^2} -199\,584\, u &= f(x) ,\\
		\left.\frac{\rd u}{\rd x}\right|_{x=\pm1} = \left.\frac{\rd^2 u}{\rd x^2}\right|_{x=\pm1} = \left.\frac{\rd^5 u}{\rd x^5}\right|_{x=\pm1} &= 0,\label{eq:model_even_2:bc}
	\end{align}
where
	\begin{equation}
		f(x)=-199\,584 x^{12} + 465\,696 x^{10} - 574\,560 x^ 4+ 50\,1984 x^2 - 147\,456 .
		\label{eq:model_even_2:rhs} 
	\end{equation}%
\end{subequations}
Problem~\eqref{eq:model_even_2} admits the same exact solution~\eqref{eq:model_even_1:exact} as problem~\eqref{eq:model_even_1}. For the same reasons as above, we only use the {even} eigenfunctions in the spectral expansion for this problem.

\begin{figure}
	\centering
	\begin{subfigure}[b]{0.4\textwidth}
		\centering
		\includegraphics[width=\textwidth]{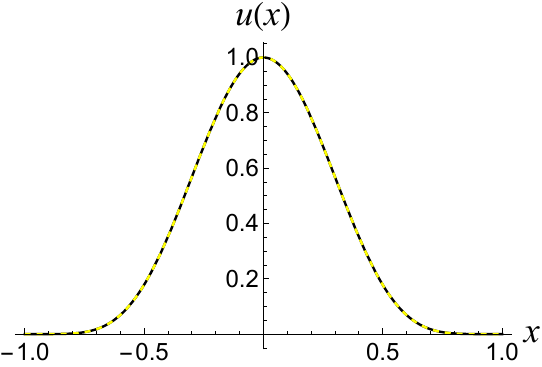}
		\caption{\label{fig:model_even_2_shape}}
	\end{subfigure}
	\qquad
	\begin{subfigure}[b]{0.4\textwidth}
		\centering
		\includegraphics[width=\textwidth]{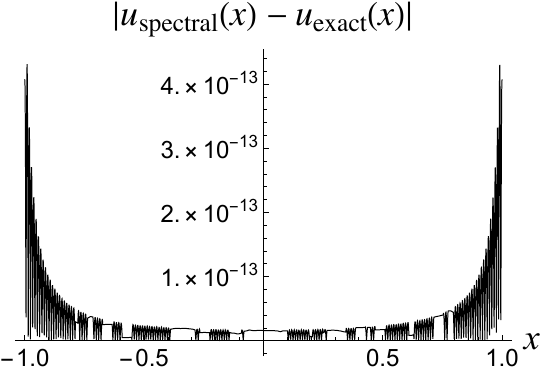}
		\caption{\label{fig:model_even_2_error}}
	\end{subfigure}
	\\
	\begin{subfigure}[b]{0.47\textwidth}
		\centering
		\includegraphics[width=\textwidth]{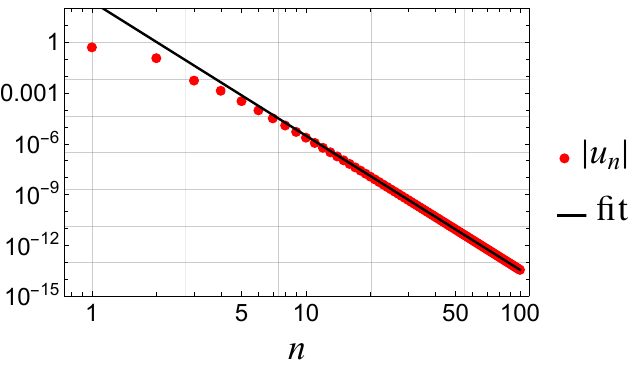}
		\caption{\label{fig:model_even_2_coeff_decay}}
	\end{subfigure}
	\caption{(a) Comparison of the exact~\eqref{eq:model_even_1:exact} (solid) and spectral expansion~\eqref{eq:u_x_expansion_formula}, \eqref{eq:model_2_u0}, \eqref{eq:lin_sys_model_even_2} (dashed) solution profiles for model problem II with $M=100$. (b) Absolute error between the exact solution and the spectral expansion. (c) Convergence rate of the spectral expansion with symbols showing $|u_n^c|$ and the solid line being the best fit $259 n^{-7.94}$  based on values for $n\ge50$.~\label{fig:model_even_2}}
\end{figure}

We again use expansion~\eqref{eq:u_x_expansion_formula}, substitute it into \eqref{eq:model_even_2:de}, take successive inner products with $\psi_0(x)=1$ and $\big\{\psi_{m}^c(x)\big\}_{m=1}^{M}$, and employ the orthogonality relations, as in the derivation of \eqref{eq:dynamical_system}. First, the coefficient $u_0^c$ is found to be
\begin{equation}
	- 199\,584 u_0^c = \int_{-1}^{+1} f(x) \psi_0(x) \,\rd x \quad \Rightarrow \quad u_0^c = \frac{2\,048}{3\,003}.
	\label{eq:model_2_u0}
\end{equation}
Then, the remaining coefficients $\{u_n^c\}_{n=1}^M$ are obtained as the solution of the \emph{linear} system of equations:
\begin{multline}
	\label{eq:lin_sys_model_even_2}
	\sum_{n=1}^M \left\{\left[-199\,584-(\lambda^c_{n})^6\right]\delta_{nm} -5\,544 \beta_{nm}^c \right\} u_n^c \\ = 
	-199\,584 \chi^{\{12\}}_{m} + 465\,696 \chi^{\{10\}}_{m} - 574\,560 \chi^{\{4\}}_{m} + 50\,1984 \chi^{\{2\}}_{m},\qquad m=1,2,\hdots,M.
\end{multline}
The coefficient matrix
\begin{equation}
	\label{eq:coeff_matrix_model_even_2}
	\mathbf{A}=[a_{nm}],\quad \text{where}\quad a_{nm} = \left[-199\,584-(\lambda^c_{n})^6\right]\delta_{nm} -5\,544 \beta_{mn}^c , \qquad n,m=1,2,\ldots,M,
\end{equation}	
of the linear system~\eqref{eq:lin_sys_model_even_2} is now full, but \emph{symmetric negative definite} (with $\|\mathbf{A}^{-1}\|<1$). Thus, system~\eqref{eq:lin_sys_model_even_2} can be easily and efficiently solved using the $\mathbf{L}\mathbf{D}\mathbf{L}^\top$ factorization.

The exact solution~\eqref{eq:model_even_1:exact} and the spectral approximation~\eqref{eq:u_x_expansion_formula}, \eqref{eq:model_2_u0}, \eqref{eq:lin_sys_model_even_2} using the sixth-order eigenfunctions are compared in figure~\ref{fig:model_even_2_shape}. Again, as seen in figure~\ref{fig:model_even_2_error}, the spectral expansion is \emph{highly} accurate with the maximum absolute error being again $\le 5\times 10^{-13}$. Model problem I and II have the same exact solution \eqref{eq:model_even_1:exact}, so the spectral expansions in this and the previous subsection are of the same function, even if the BVPs are different. Thus, it should be no surprise that figures~\ref{fig:model_even_1_error} and \ref{fig:model_even_2_error} look almost identical.
Note that, as shown in figure~\ref{fig:model_even_2_coeff_decay}, the spectral expansion has rapid convergence, with the coefficients decaying in magnitude as $\mathrm{O}(n^{-8})$, \emph{despite} the moderate convergence rate of $\beta_{nm}$ demonstrated in figure~\ref{fig:even_beta5m_coeff}. This observation further emphasizes the motivation for the proposed approach.

\section{Conclusions}\label{sec:Conclusions}

We proposed an {efficient} and {highly accurate} Galerkin spectral method for the solution of sixth-order boundary-value problems arising in the study of elastic-plated thin films. We explicitly constructed the sixth-order eigenfunctions and appealing to classical results from Sturm--Liouville theory, we showed that these functions form complete orthonormal bases of even and odd functions for $L^2[-1,1]$. The considered boundary conditions were dictated by a specific physical problem under consideration, and they led to a self-adjoint eigenvalue problem. Accurate asymptotic formulas for the corresponding eigenvalues were found. Further, we developed exact expansion formulas for the second and fourth derivatives of the eigenfunctions and for even powers of the independent spatial variable, $x$.  

The proposed spectral method was tested by solving two model problems using 101 terms ($M=100$) in the spectral expansion. In both cases, the absolute error was less than $5\times10^{-13}$ and the convergence rate of the spectral series was found to be eighth-order algebraic ($\mathrm{O}(n^{-8})$, where $n$ is the number of terms), which is even greater than the order (sixth) of the EVP.
    
In future work, we would like to consider the unsteady IBVP and develop time-stepping schemes (say, of Crank--Nicolson type) for the proposed Galerkin method. In addition, we would like to consider non-self-adjoint versions of EVP~\eqref{eq:6th_EVP}, which can arise when both bending and tension are important ($\mathscr{T}\ne0$) and/or due to the importance of gravity (when the elastic Bond number $\mathscr{B}\ne0$, see \cite{Gabay2023}).


\ack ICC would like to acknowledge the hospitality of the University of Nicosia, Cyprus, where this work was completed thanks to a Fulbright US Scholar award from the US Department of State, and the US National Science Foundation, which supports his research on interfacial dynamics under grant CMMI-2029540.

\appendix

\section*{Appendix}

\setcounter{section}{1}

The expressions for the coefficients with $m\ge1$ in the expansion formulas~\eqref{eq:even_powers_x_expansion} for the even powers of $x$ are:
\begin{subequations}\label{eq:even_powers_x_coeff_exp}
	\begin{align}
		\chi^{\{2\}}_m&= -\tfrac{4\,c_m^c}{(\lambda_m^c)^3 (\cos{\lambda_m^c} -\cosh{\sqrt{3} \lambda_m^c})} \times  \left[-\lambda_m^c \cos{2 \lambda_m^c} + \sqrt{3} \lambda_m^c \sin{\lambda_m^c} \sinh{\sqrt{3}\lambda_m^c} \right.\label{eq:chi_sq_exp} \\
		& \phantom{=}\ + \left.3\sin{\lambda_m^c}\cos{\lambda_m^c} + (\lambda_m^c \cos {\lambda_m^c} - 3 \sin{\lambda_m^c}) \cosh{\sqrt{3}\lambda_m^c}\right]\,,\nonumber \\
		\chi^{\{4\}}_m&=\tfrac{8\,c_m^c}{(\lambda_m^c)^4 (\cos{\lambda_m^c} -\cosh{\sqrt{3} \lambda_m^c})} \times  \left[ 
		((\lambda_m^c)^2-6) (\cos {2\lambda_m^c} - \cos{\lambda_m^c}\cosh{\sqrt{3}\lambda_m^c}) \right.\label{eq:chi_quad_exp} \\
		&\phantom{=}\ -\left.  \left(\sqrt{3} ((\lambda_m^c) ^2+6) \sinh{\sqrt{3}\lambda_m^c} + 9 \lambda _m^c(\cos{\lambda_m^c} - \cosh{\sqrt{3}\lambda_m^c})\right) \sin{\lambda_m^c} \right],  \nonumber \\
		\chi^{\{6\}}_m&= \tfrac{6\,c_m^c}{(\lambda_m^c)^6 (\cos{\lambda_m^c} -\cosh{\sqrt{3} \lambda_m^c})} \times  \bigg[ 360 + 2 ((\lambda_m^c)^4 - 20 (\lambda_m^c)^2 - 60) \cos{2\lambda_m^c}   \label{eq:chi_six_exp} \\
		 &\phantom{=}\ - 15(\lambda_m^c)^3 \sin{2\lambda_m^c} + \Big(30 (\lambda_m^c)^3 \sin{\lambda_m^c} - 2((\lambda_m^c)^4 - 20 (\lambda_m^c)^2+120) \cos{\lambda_m^c}\Big) \cosh{\sqrt{3}\lambda_m^c} \nonumber \\ 
		 &-2 \sqrt{3} (\lambda_m^c)^2 ((\lambda_m^c)^2 + 20) \sin{\lambda_m^c} \sinh{\sqrt{3}\lambda_m^c}  \bigg], \nonumber \\
	     \chi^{\{8\}}_m&= \tfrac{8\,c_m^c}{(\lambda_m^c)^9 (\cos{\lambda_m^c} -\cosh{\sqrt{3} \lambda_m^c})} \times \bigg[2520 (\lambda_m^c)^3   - 21 ((\lambda_m^c)^6-720) \sin{2\lambda_m^c} \\
	     &\phantom{=}\  + 2 ((\lambda_m^c)^6-42(\lambda_m^c)^4 - 420(\lambda_m^c)^2 - 5040) \lambda_m^c  \cos{2\lambda_m^c} \nonumber \\ 
	     & - 2 \sqrt{3}\lambda_m^c ((\lambda_m^c)^6 + 42 (\lambda_m^c)^4 - 5040) \sin{\lambda_m^c}\sinh{\sqrt{3}\lambda_m^c}  \nonumber \\ 
	      &\phantom{=}\  + 2 \Big(21 ((\lambda_m^c)^6-720) \sin{\lambda_m^c} - \lambda_m^c((\lambda_m^c)^6 - 42 (\lambda_m^c)^4 + 840 (\lambda_m^c)^2 - 5040) \cos{\lambda_m^c}\Big) \cosh{\sqrt{3}\lambda_m^c} 
	      \bigg], \nonumber \\
	    \chi^{\{10\}}_m&= \tfrac{20\,c_m^c}{(\lambda_m^c)^{10} (\cos{\lambda_m^c} -\cosh{\sqrt{3}\lambda_m^c})} \times  \bigg[  4536 (\lambda_m^c)^4-\tfrac{27}{2} ((\lambda_m^c)^6-20160) \lambda_m^c  \sin{2\lambda_m^c}   \\
	    & +  ((\lambda_m^c)^8 - 72 (\lambda_m^c)^6 - 1512(\lambda_m^c)^4 - 60480(\lambda_m^c)^2 + 362880) \cos{2\lambda_m^c}  \nonumber \\
	       & -\sqrt{3}((\lambda_m^c)^8 + 72 (\lambda_m^c)^6 - 60480 (\lambda_m^c)^2 - 362880) \sin{\lambda_m^c}\sinh{\sqrt{3}\lambda_m^c} \nonumber \\
	          & + \cosh{\sqrt{3}\lambda_m^c} \cdot \bigg(27 \lambda_m^c((\lambda_m^c)^6 - 20160)\sin{\lambda_m^c} \nonumber \\
	          & - ((\lambda_m^c)^8 - 72 (\lambda_m^c)^6 + 3024(\lambda_m^c)^4 - 60480(\lambda_m^c) ^2 + 362880) \cos{\lambda_m^c}\bigg) 
	     \nonumber \bigg],\\
		 \chi^{\{12\}}_m&= \tfrac{12\,c_m^c}{(\lambda_m^c)^{10} (\cos{\lambda_m^c} -\cosh{\sqrt{3} \lambda_m^c})} \times  \bigg[  23760 ((\lambda_m^c) ^6-5040) - 33(\lambda_m^c)^3((\lambda_m^c)^6-151200)\sin{2\lambda_m^c} \label{eq:twe_six_exp} \\ 
		& + 2 ((\lambda_m^c) ^{10}-110(\lambda_m^c)^8 - 3960(\lambda_m^c)^6 - 332640 (\lambda_m^c)^4  + 6652800 (\lambda_m^c) ^2 + 19958400) \cos{2\lambda_m^c} \nonumber \\
		& -2 \sqrt{3} (\lambda_m^c)^2 ((\lambda_m^c)^8 + 110 (\lambda_m^c)^6 - 332640 (\lambda_m^c)^2-6652800) \sin{\lambda_m^c} \sinh{\sqrt{3}\lambda_m^c} \nonumber \\
		 &  + 2\cosh{\sqrt{3}\lambda_m^c} \cdot \bigg( 33 (\lambda_m^c)^3 ((\lambda_m^c)^6-151200)\sin{\lambda_m^c} \nonumber \\
		 & - \left((\lambda_m^c)^{10} - 110 (\lambda_m^c)^8 + 7920 (\lambda_m^c)^6 - 332640 (\lambda_m^c)^4 + 6652800 (\lambda_m^c) ^2 - 39916800\right) \cos {\lambda_m^c} \bigg) \bigg].   \nonumber
	\end{align}
\end{subequations}   	

\section*{References}
\bibliography{Sixth-Order.bib}

\end{document}